\documentclass[12pt,leqno]{amsart}
\usepackage{a4wide}
\usepackage{amssymb}
\usepackage{amsmath}
\usepackage{amsthm}
\usepackage{verbatim}
\usepackage{fullpage}
\usepackage[all,cmtip]{xy} 
\usepackage{yfonts}
\usepackage[final]{changes}

\usepackage[T1]{fontenc}
\usepackage[utf8]{inputenc}
\usepackage{mathtools}   
\usepackage{amsfonts}
\usepackage{mathrsfs, todonotes}
\usepackage[mathscr]{euscript}

\usepackage{hyperref}
\numberwithin{equation}{section}

\theoremstyle{plain}

\newcommand{\A}{\ensuremath{{\mathbb{A}}}}

\newcommand{\cc}{\mathfrak{c}}
\newcommand{\Z}{\ensuremath{{\mathbb{Z}}}}

\newcommand{\R}{\ensuremath{{\mathbb{R}}}}
\newcommand{\F}{\ensuremath{{\mathbb{F}}}}

\newcommand{\vv}{\ensuremath{v_{\mathfrak{A}}}}

\newcommand{\T}{\ensuremath{{\mathbb{T}}}}
\renewcommand{\L}{\ensuremath{{\mathbb{L}}}}

\newcommand{\Vol}{\text{Vol}}
\newcommand{\GL}{\ensuremath{{\text{GL}}}}
\newcommand{\PGL}{\ensuremath{{\text{PGL}}}}
\newcommand{\Char}{\ensuremath{{\text{char}}}}

\newcommand{\Image}{\ensuremath{{\text{Image}}}}
\newcommand{\diag}{\ensuremath{{\text{diag}}}}
\newtheorem{theo}{Theorem}[section]
\newtheorem{lem}[theo]{Lemma}

\newtheorem{cor}[theo]{Corollary}

\theoremstyle{remark}
\newtheorem{rem}[theo]{Remark}

\theoremstyle{definition}
\newtheorem{defn}[theo]{Definition}

\newtheorem*{cor*}{Corollary}

\newcommand{\Hom}{\operatorname{Hom}}
\newcommand{\Ind}{\operatorname{Ind}}

\newcommand{\Tr}{\text{Tr}}
\newcommand{\Nm}{\text{Nm}}

\begin{document}
\bibliographystyle{plain}
\title{Sup norm on $\PGL_n$ in depth aspect}
\author{Yueke Hu}
\address{Department of Mathematics\\
  ETH Zurich\\
  Zurich\\
  Switzerland}
\email{huyueke2012@gmail.com}

 \begin{abstract}
In this short paper we give the sub-local upper bound for the sup norm of an automorphic form on $\PGL_n$, whose associated automorphic representation has finite conductor $C(\pi)=p^c$ with $c\rightarrow \infty$, and its local component at the place of ramification is a minimal vector belonging to an irreducible representation with generic induction datum.
 \end{abstract}

\maketitle
\tableofcontents

\section{Introduction}
\subsection{Sup norm problem}
The study of sup norm of an automorphic form restricted to a compact domain originated from the work of \cite{iwaniec_$l^infty$_1995}, and has since be extended to various aspects and groups.
In this short paper we study the sup norm problem of automorphic forms on $\PGL_n$ in the depth aspect. We shall give a generalisation of our previous work in depth aspect  for $\GL_2$ in \cite{HuNelsonSaha:17a}, which at the same time is an $p-adic$ analogue of \cite{Marshall:} and \cite{BlomerMaga:} in the archimedean aspect, in a simple and explicit fashion. Instead of the standard setting in which people usually choose newforms as the local component at ramified places, we choose the so-called minimal vectors, whose special properties allow us to greatly simplify the computations and estimates. We would like to use this result to advocate the idea that minimal vectors are better or more natural choices than classical newforms for studying certain algebraic and analytic problems in depth aspect.

More specifically, let $\pi$ be an automorphic representation on $\PGL_n$ over any number field with bounded archimedean parameter and finite conductor $p^c$ with ramification at a single place $v$ and $c\rightarrow \infty$. Furthermore we assume that $\pi_v$ has generic induction datum. See Definition \ref{Defn:ParabolicGeneric} for the exact meaning for such representations. 
For supercuspidal representations, this implies that the associated datum
 $\beta$ is 'minimal' (which is not related to 'minimal' vector) as in Definition \ref{Defn:minimalBeta}. For parabolically induced representation, this implies that the all supercuspidal representations $\sigma_j$ used in the construction have generic induction datum,  are not related, and have similar normalised depth $c(\sigma_j)$ (see Definition \ref{Def:Depthnomalised}).
These are the p-adic analogues of the principal series representations at archimedean places whose Satake parameters stay away from the walls of Weyl chambers. 
Let $F\in \pi$ be an automorphic form whose local components at  unramified places are  spherical, and at $v$  is a minimal vector as in Definition \ref{Defn:minimalForParabolic}.

Under such assumptions, our main result is the following. 

\begin{theo}For an automorphic form $F$ as above and $g$ in a fixed compact domain $\Omega$ of $\PGL_n(\A)$, we have
\begin{equation}
\sup|F(g)|\ll_{\Omega, n, p} C(\pi)^{\frac{n-1}{4}-\frac{1}{8n^3}},
\end{equation}
where $C(\pi)=p^c$ is the finite conductor of $\pi$.
\end{theo}

The basic strategy is to use amplified pre-trace formula following the original paper \cite{iwaniec_$l^infty$_1995}, which was also used in most subsequent papers on sup norm problem and we will directly borrow it. In particular we borrow the amplifier from \cite{BlomerMaga:}, which is the only reason we work with $\PGL_n$ instead of $\GL_n$, and can be easily extended to $\GL_n$ with additional pages. The main innovation of this paper lies in the idea to apply minimal vectors into this mechanism for more general groups.

\subsection{Minimal vector}
This result is a direct generalisation of \cite{HuNelsonSaha:17a}, where we used minimal vectors for a special family of supercuspidal representations on $\GL_2$ to obtain nice upper and lower bounds for the sup norm of automorphic forms. Such test vectors arise naturally from compact induction theory in general, but were rarely used to study analytic number theory problems.  In a recent paper \cite{HN18} we also use them to study the test vector problem for Waldspurger's period integral, and can effectively reprove the dichotomy and Tunnell-Saito's $\epsilon-$value test along the way whenever a Lie-algebra argument can apply. 

This paper shows that one can easily generalise the minimal vector of  \cite{HuNelsonSaha:17a} to a more general setting by making use of the compact induction theory on $\GL_n$ in \cite{BushnellKutzko:a}, and conveniently get a nontrivial bound for the sup norm. Such test vector can be uniquely identified by its property as in Corollary \ref{Cor:Minimal&Unique}. The main ingredient for getting sup norm upper bound is an explicit description of the matrix coefficient for the minimal vector, which serves two purposes. 
\begin{enumerate}
\item[(I)] It allows explicit computations for the formal degree, which can be related to the finite conductor and directly give a convexity bound for sup norm when using matrix coefficient as test function for the pre-trace formula.
\item[(II)] One can also see that the support of its matrix coefficient is concentrated in a neighbourhood of a torus, and we shall exploit this property to get an easy lattice point counting for the geometric side of pre-trace formula and get a power saving.
\end{enumerate}

We shall also go beyond supercuspidal representations by giving similar test vectors for parabolic inductions, which directly generalise the test vector used in \cite{nelson_microlocal_2016} and is closely related to the semisimple type in  \cite{BushnellKutzko:}. 
We shall choose the test function to be the restriction of the matrix coefficient of such minimal vector to proper compact subgroup, while maintaining similar properties (I) (II) as above. Whether such minimal vector is unique and whether the test function is capturing most information (meaning that the compact subgroup we chose is of bounded index in the support of the matrix coefficient when restricted to a maximal compact open subgroup) are  trickier questions, which we shall discuss in Remark \ref{Rem:UniquenessParabolic} and \ref{Rem:testfuncFormaldegree}. Under the assumption that $\pi$ has generic induction datum, the result in this case is consistent with the supercuspidal representation case.

We do remark that it's also possible to obtain sub-local bound for newforms on  $\GL_2$ in depth aspect, which is an ongoing project of the author with Abhishek Saha. It however requires more complicated discussion on the associated matrix coefficient, and doesn't seem easily generalisable to $\GL_n$.

\subsection{Comparison}
One can compare this result with the sup norm bound in \cite{BlomerMaga:}, where they obtained
\begin{equation}
\sup |F(g)|\ll \lambda^{\frac{n(n-1)}{4}(\frac{1}{2}-\delta)}.
\end{equation}
Here $\lambda$ is the Laplace eigenvalue for $F$.
Suppose that 
for $\pi_\infty$, the Satake parameters $\mu_i$ for $i=1,\cdots, n$ stay away from the walls of Weyl chambers, and in particular are of similar size. Then $\lambda\asymp_n |\mu_i|^2\asymp C(\pi)^{\frac{2}{n}}$ where $C(\pi)$ is the analytic conductor. Then the sup norm bound in eigenvalue aspect can be rewritten as
\begin{equation}
\sup |F(g)|\ll C(\pi)^{\frac{n-1}{4}-\delta'}.
\end{equation}
Thus our result is the exact analogue of \cite{BlomerMaga:} in the p-adic setting for generic induction datum. Our power saving is explicit and better due to the fact that we are exploiting the special properties of the minimal vector at ramified  places.

\subsection{Further possible generalisations}
For the sake of simplicity and conciseness, we prove our main result in the setting of representations of $\PGL_n$ with generic induction datum. Here we briefly discuss possible generalisations.

The strategy should also work for more general representations other than those with generic induction datum. Minimal vectors can still be defined in a similar fashion. The main differences are directly related to properties [I] and [II] for minimal vectors.
For [I], while one can still effectively estimate the 'formal degree',  there is no longer a simple relation between the normalised depth and the finite conductor. One would need to formulate the result in terms of more detailed information of the induction datum. 

For [II],  it is possible that the matrix coefficient of the minimal when restricted to a maximal compact subgroup is concentrated around a non-trivial Levy subgroup, instead of a simple torus. In that case, the lattice point counting problem would be more complicated (while still possible regarding \cite{Marshall:}). One can also avoid this issue by assuming that, for example for parabolically-induced representations, there exists a constant $\delta>0$ such that each normalised depth $c(\sigma_j)\geq \delta c(\pi)$. Similar condition for supercuspidal representations exists, and then the associated matrix coefficient is still concentrated around a torus.

In principle the method in this paper will also work for sup norm problems on classical groups regarding the similar construction of supercuspidal representations on them in \cite{stevens_supercuspidal_2008}.

As explained in  \cite{HuNelsonSaha:17a} these minimal vectors are the $p-adic$ analogue of micro-local lifts. It would also be interesting to see if one can easily get a similar upper bound  for micro-local lifts in the archimedean aspect.

\subsection{Acknowledgement}
The author would like to thank Simon Marshall, Philippe Michel, Paul Nelson, Abhishek Saha, Will Savin and Nicolas Templier for helpful discussions. 
This work is partially supported by NSF Grant No. DMS-1440140 during author's stay at MSRI in 2017 spring, and later on supported by SNF grant SNF-169247.

\section{Preliminary notations and results}

 Let $\lfloor x\rfloor$ denote the floor function for $x\in \R$, and $\lceil x\rceil$ denote the ceiling function.

Let $M_{n_1\times n_2}$ the $n_1\times n_2$ matrix algebra.  $M_{n}=M_{n\times n}$.

For a p-adic field $\F_v$, let $k_{\F_v}$ be its residue field, $\varpi_v$ be a uniformiser, $O_v$ be the ring of integers. $K_v$ is the maximal compact open subgroup of $\GL_n(\F_v)$. Let $\psi_v$ be an additive character of level 1 (instead of an unramified additive character in convention).

By abuse of notations subgroups which are compact-mod-center will also be referred to as compact subgroups.

\begin{defn}
Let $\pi$ be an irreducible representation of $G$, $H\subset G$ is a compact open subgroup and $\rho$ is an irreducible representation of $H$. $\pi$ is said to contain $\rho$ if $\Hom_{H}(\rho, \pi)\neq 0$.
\end{defn}

\begin{defn}
For $i=1,2$ let $H_i$ be compact open subgroups of $G$ and $\rho_i$ be irreducible representations of $H_i$. We say $g\in G$ intertwines $\rho_1$ with $\rho_2$ if
$$\Hom_{H_1\cap H_2^g}(\rho_1,\rho_2^g)\neq 0.$$
Here $H_2^g=g^{-1}H_2 g$ and $\rho_2^g$ is the representation of $H_2^g$ given by $\rho_2^g(g^{-1}hg)=\rho_2(h)$. In the case $H_1=H_2=H$ and $\rho_1=\rho_2=\rho$, we simply say $g$ intertwines $\rho$.
\end{defn}
\begin{lem}\label{Lem:alwaysIntertwine}
Let $\pi$ be an irreducible smooth representation of $G$. Let $\rho_i$ be irreducible smooth representations of $H_i$ which are  compact open subgroups of $G$. Then there exists $g\in G$ which intertwines $\rho_i$.
\end{lem}
\begin{lem}[Mackey theory for compact induction]	Let $K$, $H$ be compact subgroups of $G$ and $\Lambda$ be an irreducible smooth representation of $K$. Then
\begin{equation}
c-\Ind_{K}^G\Lambda|_{H}=\bigoplus\limits_{g\in K\backslash G/H} \Ind_{H\cap K^g}^H( \Lambda^g|_{H\cap K^g})
\end{equation}
\end{lem}
\begin{proof}
Similar to the proof for Mackey theory for finite groups, as one can give a basis for compact induction explicitly.
\end{proof}

\begin{lem}
Let $\pi=c-\Ind_J^G \Lambda$ be compactly induced representation for some irreducible representation $\Lambda$ of a compact open subgroup $J$. If $g$ intertwines $\Lambda$ iff $g\in J$, then $\pi$ is irreducible and supercuspidal.
\end{lem}

\section{Minimal vector}
This section is purely local and we shall omit sub-script $v$ from all notations.
\subsection{Supercuspidal representation and minimal vector}
Here we follow \cite{BushnellKutzko:a} closely to review useful results for the construction of supercuspidal representations on $\GL_n$. 
\begin{defn}\label{Defn:standardHerediatoryorder}
For a fixed interger $e=e_\mathfrak{A}$ dividing $n$ and $m=\frac{n}{e}$, a principal hereditary order $\mathfrak{A}$ in $M_n$ is up to a conjugation of the following shape
\begin{equation}\label{Eq:standardHerediatoryorder}
\mathfrak{A}= \left(\begin{array}{cccc}
M_m(O_\F)&M_m(O_\F) &\cdots & M_m(O_\F)\\
\varpi M_m(O_\F) &M_m(O_\F) &\cdots &M_m(O_\F)\\
\vdots &\vdots &\ddots &\vdots \\
\varpi M_m(O_\F) &\varpi M_m(O_\F)&\cdots &M_m(O_\F)
\end{array}\right).
\end{equation}
Its Jacobson radical $\mathcal{B}$ is then of shape
\begin{equation}
\mathcal{B}= \left(\begin{array}{cccc}
\varpi M_m(O_\F)&M_m(O_\F) &\cdots & M_m(O_\F)\\
\varpi M_m(O_\F) &\varpi M_m(O_\F) &\cdots &M_m(O_\F)\\
\vdots &\vdots &\ddots &\vdots \\
\varpi M_m(O_\F) &\varpi M_m(O_\F)&\cdots &\varpi M_m(O_\F)
\end{array}\right).
\end{equation}
It has the property that
\begin{equation}
\mathcal{B}^e=\varpi \mathfrak{A}.
\end{equation}
Denote $\mathfrak{A}_0$ to be the  hereditary order $M_n(O_\F)$ 
with $e_\mathfrak{A}=1$, and $\mathcal{B}_0$ to be its associated Jacobson radical $\varpi M_n(O_\F)$.
\end{defn}
Thus by mutiplying with a suitable negative power of $\varpi$ one can define $\mathcal{B}^i$ for $i<0$.
This gives rise to a filtration of orders $\{\mathcal{B}^i\}_{i\in \Z}$, which defines a semi-valuation $v_\mathfrak{A}$ on $M_n$
$$v_\mathfrak{A}(x)=\min\limits_{i}\{x\in \mathcal{B}^i\}.$$
One can directly verify the following approximation result.
\begin{cor}\label{Cor:herediatoryorderApproximation}
Let $\mathcal{B}$ be  in the standard form as in Definition \ref{Defn:standardHerediatoryorder}, and $\mathcal{B}_0$ be as above. Then for any $i\in \Z$,
\begin{equation}
\mathcal{B}_0^{\lceil \frac{i-1}{e_\mathfrak{A}}\rceil+1}\subset \mathcal{B}^i\subset \mathcal{B}_0^{\lfloor \frac{i}{e_\mathfrak{A}}\rfloor}.
\end{equation}
\end{cor}

\begin{defn}\label{Defn:maximalbeta}
Let $\beta\in M_n$ with $v_\mathfrak{A}(\beta)=-j<0$ be such that $\L=\F[\beta]$ is a field of degree $n$ (which is maximal possible) with the property that $\L^*$ normalises $\mathfrak{A}$ and its ramification index $e_{\L/\F}=e_\mathfrak{A}$.
\end{defn}

 $v_\mathfrak{A}$  is a semi-valuation  with respect to $\L$ in the sense that 
\begin{equation}\label{Eq:SemiValuation}
 \vv(lx)=v_\L(l)\vv(x), \forall l\in \L^*,
\end{equation} 
where $v_\L(x)$ is the normalised p-adic evaluation on the local field $\L$.

Let $\alpha_\beta$ denote the map
\begin{equation}
\alpha_\beta(x)=\beta x- x\beta.
\end{equation}
Define an interger 
\begin{equation}\label{Eq:defnOFk0}
 k_0(\beta,\mathfrak{A})=\max\limits_{i}\{\exists x\in \mathfrak{A}, x\notin \mathcal{B}+O_\L, \alpha_\beta(x)\in \mathcal{B}^k\}.
\end{equation}
\begin{defn}\label{Defn:minimalBeta}
$\beta$ is called 'minimal' over $\F$ in  \cite{BushnellKutzko:a} if it satisfies
\begin{enumerate}
\item $v_\L(\beta)$ is coprime to $e_{\L/\F}$, and
\item $\varpi^{-v_\L(\beta)}\beta^{e_{\L/\F}}+\varpi_\L O_\L$ generates the residue field extension $k_\L/k_\F$. 
\end{enumerate}

\end{defn}

Note here $v_\L(\varpi^{-v_\L(\beta)}\beta^{e_{\L/\F}})=0$. $\beta$ being minimal over $\F$ is equivalent to that $$k_0(\beta,\mathfrak{A})=v_\mathfrak{A}(\beta).$$
$\beta$ being minimal is the analogue of the condition for principal series representations that the Satake parameters stay away from the walls of Weyl chamber, where the Galois group becomes the Weyl group.

While supercuspidal representations can be associated to general $\beta$ as in Definition \ref{Defn:maximalbeta}, we shall focus on those associated to the minimal ones, in which case the construction is slightly easier. 
\begin{defn}
Suppose that $\beta$ is minimal.
Let $U_\mathfrak{A}(i)=1+\mathcal{B}^i$.
Let $H^1=H^1(\beta,\mathfrak{A})=U_\L(1)U_\mathfrak{A}(\lfloor\frac{j}{2}\rfloor+1)$, $J^1=J^1(\beta,\mathfrak{A})=U_\L(1)U_\mathfrak{A}(\lfloor\frac{j+1}{2}\rfloor)$, $J=J(\beta,\mathfrak{A})=\L^*U_\mathfrak{A}(\lfloor\frac{j+1}{2}\rfloor)$. Note that $\lfloor\frac{j}{2}\rfloor+1=\lceil \frac{j+1}{2}\rceil$.
\end{defn}
Note that $J^1=H^1$ when $j=-v_\mathfrak{A}(\beta)$ is odd. Otherwise $J^1/H^1$ is a $n^2-n$ (which is alway even) dimensional  vector space over the residue field $k_\F$. 
\begin{defn}
A simple character of $H^1$ associated to $\beta$ is defined to be any character $\theta$ on $H^1$ such that
\begin{equation}
\theta |_{U_\mathfrak{A}(\lfloor\frac{j}{2}\rfloor+1)}(x)=\psi\circ \Tr_{M_n}(\beta(x-1)).
\end{equation}
\end{defn}
In particular $\theta$ is trivial on $U_\mathfrak{A}(j+1)$.
\begin{rem}
When $p$ is sufficiently large, one can also require that $$\theta|_{U_\L(1)}(x)=\psi\circ \Tr_{\L/\F}(\beta \log_p(x)).$$ For this definition, the simple character $\theta$ associated to $\beta$ is then unique. For the purpose of this paper, we shall just use the previous definition.
\end{rem}
The most important property for a simple character is about its intertwiner
\begin{lem}\label{Lem:intertwining&conjugacy}
$g$ intertwines $\theta$ iff $g\in J$ iff $g$ conjugates $\theta$ (i.e., $g$ intertwines $\theta$ and $(H^1)^g=H^1$).
\end{lem}

The next step is to extend a simple character to be a representation of $J^1$ in the case $J^1\neq H^1$.
\begin{lem}
There exists a unique irreducible representation $\eta(\theta)$ of $J^1$ such that $\eta(\theta)|_{H^1}$ contains $\theta$. Moreover $\eta(\theta)|_{H^1}$ is a multiple of $\theta$ and $$\dim (\eta(\theta))=(J^1:H^1)^{1/2}.$$
\end{lem}
When $J^1=H^1$ this lemma is trivial. Otherwise $\eta(\theta)$ is actually a Heisenberg representation constructed as follows. Let $B^1$ be  any intermediate group between $H^1$ and $J^1$ which polarises $J^1/H^1$ under pairing given by $(x,y)\mapsto \psi\circ \Tr_{M_n}(\beta xy)$, and $\tilde{\theta}$ be any extension of $\theta$ to $B^1$. Then 
\begin{equation}\label{Eq:HeisenbergExt}
 \eta(\theta)=\Ind_{B^1}^{J^1}\tilde{\theta}.
\end{equation}
The general theory of Heisenberg representation tells that $\eta(\theta)$ is independent of choices of $B^1$ and $\tilde{\theta}$. For uniformity we denote $B^1=H^1$ and $\tilde{\theta}=\theta$ in the case when $J^1=H^1$.

\begin{lem}
There exists a representation (usually not unique) $\Lambda$ of $J$ called simple type such that $\Lambda|_{J^1}=\eta(\theta)$. In particular $\dim \Lambda=\dim \eta(\theta)$. $g$ intertwines $\Lambda$ iff $g\in J$.
\end{lem}

Then immediately the representation $\pi=c-\Ind_{J}^{G}\Lambda$ is irreducible and supercuspidal. 
\begin{defn}\label{Def:Depthnomalised}
We call $$d(\pi)=-v_\mathfrak{A}(\beta)$$ the depth of the supercuspidal representation $\pi$ (which is always an integer), and $$c(\pi)=-v_\mathfrak{A}(\beta)/e_{\mathfrak{A}}$$ the normalised depth of $\pi$ (which may not be an integer).
\end{defn}

\begin{defn}\label{Defn:genericSupercuspidal}
In general any supercuspidal representation arise in a similar way for proper $\beta$ as in Definition \ref{Defn:maximalbeta}.
To avoid confusion between the minimal element $\beta$ and minimal test vector, we shall call supercuspidal representations arising from a minimal $\beta$ a supercuspidal representation with \textit{generic induction datum}.
\end{defn}

\begin{cor}\label{Cor:Minimal&Unique}
Let $\pi$ be a supercuspidal representation with generic induction datum.
Let $\varphi_0\in \pi$ be a vector on which
$B^1$ acts by $\tilde{\theta}$. Then such vector exists and is unique up to a constant.

Furthermore for any $g\in G$, $\pi(g)\varphi_0$ is the unique vector on which $(B^1)^g$ acts by $\tilde{\theta}^g$. We call any of these vectors a minimal vector.
\end{cor}
\begin{proof}
The existence follows directly from the construction \eqref{Eq:HeisenbergExt} and compact induction. The uniqueness comes from Mackey theory and Lemma \ref{Lem:intertwining&conjugacy}. In particular one need to prove that $\Hom_{B^1}(c-\Ind_J^G\Lambda, \tilde{\theta})$ is $1-$dimensional. Using Mackey theory for compact induction, we have that
\begin{align}
 \Hom_{B^1}(c-\Ind_J^G\Lambda, \tilde{\theta})&=\sum\limits_{g\in J\backslash G/ B^1}\Hom_{B^1}(\Ind_{J^g\cap B^1}^{B^1}\Lambda^g|_{J^g\cap B^1},\tilde{\theta})\\
 &=\sum\limits_{g\in J\backslash G/ B^1}\Hom_{J^g\cap B^1}(\Lambda^g,\tilde{\theta}).\notag
\end{align}
Here we have used Frobenius reciprocity in the last equality. Let $g\in J\backslash G/ B^1$ be such that $\Hom_{J^g\cap B^1}(\Lambda^g,\tilde{\theta})$ is not zero. As $\Lambda|_{H^1}$ is a multiple of $\theta$, this in particular implies that $g$ intertwines $\theta$ with itself when restricting $\Lambda$ and $\tilde{\theta}$ to $H^1$. By Lemma \ref{Lem:intertwining&conjugacy} we have $g\in J$, corresponding to a single double $J-B^1$ coset. Taking $g=1$, we see that the corresponding space $\Hom_{J^g\cap B^1}(\Lambda^g,\tilde{\theta})=\Hom_{B^1}(\eta(\theta),\tilde{\theta})$. This is $1-$dimensional by \eqref{Eq:HeisenbergExt} because by Mackey theory and Frobenius reciprocity again,
\begin{equation}
 \Hom_{B^1}(\eta(\theta),\tilde{\theta})=\sum\limits_{g\in B^1\backslash J^1/B^1}\Hom_{B^1\cap (B^1)^g}(\tilde{\theta}^g,\tilde{\theta}),
\end{equation}
and $g\in J^1$ intertwines $\tilde{\theta}$ iff $g\in B^1$.
\end{proof}

\begin{cor}\label{Cor:MCofGeneralMinimalVec}
Let $\Phi_{\varphi_0}$ be the matrix coefficient associated to a minimal vector $\varphi_0$. Then $\Phi_{\varphi_0}$ is supported on $J$, and
\begin{equation}
\Phi_{\varphi_0}(bx)=\Phi_{\varphi_0}(xb)=\tilde{\theta}(b)\Phi(x)
\end{equation}
for any $b\in B^1$.
\end{cor}

One can approximately view $|\Phi_{\varphi_0}|$ as the characteristic function of a neighbourhood of the torus $\L^*$.

\subsection{Parabolic induction and minimal vector}
Let $M=\prod\limits_{1\leq i\leq k} \GL_{n_i}$  with $n=\sum n_i$ and let $P$ be a parabolic subgroup  of $\GL_n$ containing $M$ as  the Levi subgroup. Let $\sigma_i$ be irreducible supercuspidal representations of $\GL_{n_i}$ and $\sigma=\otimes \sigma_i$ be an irreducible supercuspidal representation of $M$. The parabolic induction theory in \cite{bernstein_induced_1977} implies that any irreducible smooth representation of $\GL_n$ can be constructed as a subquotient of parabolically-induced representation of form $\Ind_{P}^{\GL_n}\sigma$.

We shall assume the following.
\begin{defn}\label{Defn:ParabolicGeneric}
Let $\pi$ be a parabolically induced representation as above.
We call $\pi$ a representation with generic induction datum if it satisfies the following conditions.
\begin{enumerate}
\item Each $\sigma_i$ are supercuspidal representations associate to minimal $\beta_i$'s.
\item For $i\neq j$,  
$\sigma_i$ and $\sigma_j$ are not equivalent up to an unramified twist. 
In particular by \cite{bernstein_induced_1977}, $\pi=\Ind_{P}^{\GL_n}\sigma$ is irreducible.
\item $\sigma_i$'s have essentially the same normalised depth, or more precisely, let 
\begin{equation}\label{Defn:depthforParabolicInd}
 c=c(\pi)=\lceil\max\{c(\sigma_i)\}\rceil,
\end{equation}
then $c(\sigma_i)=c(\pi)+O(1)$.
\end{enumerate}
\end{defn}
These conditions again are analogues of the condition for principal series representations that the Satake parameters stay away from the walls of Weyl chamber. 
\begin{defn}
Let \begin{equation}\label{Eq:DefnofCC}
 \cc=\min\limits_{i}\{\lfloor \frac{1}{e_{\mathfrak{A}_i}}\lfloor\frac{d(\sigma_i)+1}{2}\rfloor\rfloor\}.
\end{equation} By  Definition \ref{Defn:ParabolicGeneric} (3), we have \begin{equation}\label{Eq:asypmCC}
 \cc\asymp \frac{c(\pi)}{2}.
\end{equation}
\end{defn}

By the previous section, in each $\sigma_i$, there is a test vector $\varphi_i\in \sigma_i$ such that $B^1_i$ acts on it by $\tilde{\theta}_i$. By properly translating $\varphi_i$, we further assume that each $\mathfrak{A}_i$ are taken to be the standard form given in Definition \ref{Defn:standardHerediatoryorder}.  Then each $B_i^1$ is a subgroup of the standard maximal subgroup in $\GL_{n_i}$ (as $v_{\L_i}\geq 0$ for these elements and $v_{\mathfrak{A}_i}$ is consistent with $v_{\L_i}$).
\begin{defn}
Let $\L_i$ be the torus appearing in the construction of $\pi_i$, 
$U_\L(1)=\otimes U_{\L_i}(1)$. Define the group 
\begin{equation}\label{Eq:ConstructionKpi}
 K_\pi=\left(\begin{array}{cccc}
B_1^1&\varpi^{\lfloor  \frac{c+1}{2}\rfloor}M_{n_1\times n_2}(O_\F) &\cdots & \varpi^{\lfloor \frac{c+1}{2}\rfloor}M_{n_1\times n_k}(O_\F)\\
\varpi^{\lceil  \frac{c+1}{2}\rceil}M_{n_2\times n_1}(O_\F) &B_2^1&\cdots &\varpi^{\lfloor  \frac{c+1}{2}\rfloor}O_\F\\
\vdots &\vdots &\ddots &\vdots \\
\varpi^{\lceil  \frac{c+1}{2}\rceil}M_{n_k\times n_1}(O_\F) &\varpi^{\lceil  \frac{c+1}{2}\rceil}M_{n_k\times n_2}(O_\F)&\cdots &B_k^1
\end{array}\right),
\end{equation}
\begin{cor}\label{Cor:KpiApproximation}
$K_\pi$ 
is contained in $U_\L(1)(1+\varpi^{\cc}M_n(O_\F))$.
\end{cor}
\begin{proof}
By definition and  Corollary \ref{Cor:herediatoryorderApproximation}.
\end{proof}

Define a character $\Theta=\otimes \tilde{\theta}_i$ on $K_\pi$ by 
\begin{equation}
\Theta(\left(\begin{array}{cccc}
b_1 &m_{12}&\cdots & m_{1k}\\
m_{21} &b_2 &\cdots &m_{2k}\\
\vdots &\vdots &\ddots &\vdots \\
m_{k1} &m_{k2}&\cdots &b_k
\end{array}\right))=\prod\limits_i\tilde{\theta}_i(b_i).
\end{equation}
\end{defn}
This character is well defined as
\begin{equation}
 \left(\begin{array}{cccc}
b_1 &m_{12}&\cdots & m_{1k}\\
m_{21} &b_2 &\cdots &m_{2k}\\
\vdots &\vdots &\ddots &\vdots \\
m_{k1} &m_{k2}&\cdots &b_k
\end{array}\right)\left(\begin{array}{cccc}
b_1' &m_{12}'&\cdots & m_{1k}'\\
m_{21}' &b_2' &\cdots &m_{2k}'\\
\vdots &\vdots &\ddots &\vdots \\
m_{k1}' &m_{k2}'&\cdots &b_k'
\end{array}\right)\equiv \left(\begin{array}{cccc}
b_1b_1' &0&\cdots & 0\\
0 &b_2b_2' &\cdots &0\\
\vdots &\vdots &\ddots &\vdots \\
0&0&\cdots &b_kb_k'
\end{array}\right)\mod{\varpi^{c+1}},
\end{equation}
and by definition of $c(\pi)$, each $\tilde{\theta}_i$ will be trivial on $1+\mathcal{B}_i^{d(\sigma_i)+1} \supset 1+\varpi^{c(\pi)+1}M_{n_i\times n_i}(O_\F)$ using Corollary \ref{Cor:herediatoryorderApproximation}.
\begin{cor}\label{Cor:existenceofMinimal}
Let $\pi$ be a parabolically induced representation with generic induction datum.
There exists an element $\varphi_0\in \pi$ such that $K_\pi$ acts on $\varphi_0$ by $\Theta$. 
\end{cor}
\begin{proof}
Let $\varphi_0\in \Ind_{P}^{\GL_n} \sigma$ be such that $\varphi_0$ is supported on $B K_\pi$ and defined as
\begin{equation}
\varphi_0(bk)=\sigma(b)\Theta(k)\bigotimes_i \varphi_i,
\end{equation}
where $B_i^1$ acts on $\varphi_i$ by $\tilde{\theta_i}$.
This element is well-defined following the definition of $\Theta$ and how $\sigma_i$ acts on $\varphi_i$. 
\end{proof}
\begin{defn}\label{Defn:minimalForParabolic}
For any  $\varphi_0$ as above and any $g\in G$, we call $\pi(g)\varphi_0$ a minimal vector in $\pi$.

For uniformity, we denote $K_\pi=B^1$ and $\Theta=\tilde{\theta}$ if $\pi$ is already a supercuspidal representation.
\end{defn}
\begin{rem}\label{Rem:UniquenessParabolic}
One can similarly asks for the uniqueness for the minimal vectors with fixed embedding of $K_\pi$. One can imitate the proof in the supercuspidal case. However the Mackey theory for parabolic induction is not known in general, and one need to apply  \cite{bernstein_induced_1977}[Theorem 5.2]. That theorem requires a decomposability condition, which is not true for a general pair of parabolic subgroup and compact subgroup. One can restrict the property of minimal vector to a normal compact subgroup for which the decomposability holds, then one can show that the dimension of minimal vectors with fixed embedding of $K_\pi$ can be bounded only in terms of $p$ and $n$.

\end{rem}

\subsection{Test function, formal degree, and finite conductor}
\begin{lem}\label{Lem:Conductorofpi}
Let $C(\pi)$ be the finite conductor of an irreducible smooth representation $\pi$ with generic induction datum. Let $c(\pi)$ be as in Definition \ref{Def:Depthnomalised} or \eqref{Defn:depthforParabolicInd}. Then
\begin{equation}
C(\pi)\asymp_{n,p} p^{n c(\pi)}.
\end{equation}
\end{lem}
\begin{proof}
By \cite{BushnellHenniartKutzko:}, the conductor of a supercuspidal representation $\sigma$ is
\begin{equation}
 C(\sigma)\asymp p^{nd(\pi)/e}\asymp p^{n c(\pi)}.
\end{equation}
When $\pi$ is a parabolically-induced representation with generic induction datum, we have
\begin{equation}
C(\pi)=\prod\limits_{j=1}^{k} C(\sigma_i)\asymp p^{n_i c(\sigma_i)}\asymp p^{n c(\pi)}.
\end{equation}
\end{proof}
\begin{defn}\label{Defn:formaldegTestfunCC}
Let $d_\pi=\Vol (K_\pi)$. 
Define the local test function \begin{equation}\label{Eq:DefnofTestfun}
 \omega(g)=\begin{cases}
\Theta(g), &\text{\ if }g\in K_\pi,\\
0, &\text{\ otherwise.}
\end{cases}
\end{equation}
Let $\omega^*(g)=\overline{\omega(g^{-1})}$.

\end{defn}

\begin{lem}['formal degree' of test function]\label{Cor:formaldegMC}
Let $\pi$ be an irreducible smooth representation of $\GL_n$ and $\varphi_0\in \pi$ be a minimal vector. 
 Then 
\begin{equation}
d_\pi\asymp p^{-\frac{c(n^2-n)}{2}}\asymp C(\pi)^{-\frac{n-1}{2}},
\end{equation}
and
\begin{equation}
\pi(\omega)\varphi_0=d_\pi \varpi, \text{\ } \omega*\omega^*=d_\pi \omega.
\end{equation}

\end{lem}
\begin{proof}
When $\pi$ is supercuspidal associated to minimal element $\beta$, we have
\begin{equation}
\Vol(K_\pi)=\Vol(B^1)\asymp p^{-\frac{c(n^2-n)}{2}}.
\end{equation}
In general, by construction in \eqref{Eq:ConstructionKpi} for generic induction datum we also have
\begin{equation}
\Vol(K_\pi)\asymp p^{-\frac{c(n^2-n)}{2}}.
\end{equation}
By Lemma \ref{Lem:Conductorofpi}, $C(\pi)\asymp p^{nc}$, thus 
\begin{equation}
\Vol(K_\pi)\asymp C(\pi)^{-\frac{n-1}{2}}.
\end{equation}
The rest are easy to check as $\omega$ is a character on a subgroup.
\end{proof}
\begin{rem}\label{Rem:testfuncFormaldegree}
A nature question to ask is whether the test function $\omega$ is capturing the most information of a matrix coefficient (of any test vector)
.
In the case of supercuspidal representations,  the answer is yes as the constant $d_\pi$ is asymptotically the formal degree. 

For parabolically induced representations the formal degree is not well-defined as the matrix coefficients are not square-integrable. While one can make $d_\pi$ arbitrarily large by taking $\omega$ to be the restriction of the matrix coefficient to more double $K-$cosets, the geometric side of pre-trace formula is also becoming larger. So the common sense is to take $\omega$ to be the restriction to a maximal compact open subgroup. Then the questions are what is the maximal $L^2$ norm of a matrix coefficient when restricted to a maximal compact open subgroup $K$ (or equivalently, the
 minimal dimension of a $K-$representation occurring in $\pi$), and whether one's choice of $\omega$ is essentially such a matrix coefficient. To author's knowledge, the first question is still open for general $\GL_n$. But for $\varphi_0$,   our choice of $\omega$ is indeed essentially the matrix coefficient of $\varphi_0$ restricted to $K$. This is because the support of matrix coefficient for $\varphi_0$ is controlled by the intertwining of semisimple types, which in turn is known in \cite{BushnellKutzko:}.
\end{rem}

\begin{lem}[Concentration around torus]\label{Cor:torusMC}
For any $x\in G$ with $\omega(x)\neq 0$, there exists $l\in U_\L(1)$ such that $xl^{-1}\in 1+ \varpi^\cc M_n(O_\F)$.
\end{lem}
\begin{proof}
It follows directly from the definition of $\omega$ and Corollary \ref{Cor:KpiApproximation}. Note that $l\in U_\L(1)\subset K$ normalises $1+ \varpi^\cc M_n(O_\F)$.
\end{proof}


\section{Proof of the main result}
\subsection{Amplified pre-trace inequality}
The strategy to obtain a sub-local bound for the sup norm of the automorphic forms on a group $G$ is already somewhat sophisticated. As the focus of this paper is to highlight the choice of the local test vector for the sup norm problem, we shall not seek to improve this strategy, and shall directly take the results on amplifiers from, for example, \cite{BlomerMaga:} and \cite{BlomerMaga:a}.

Recall that $L_0$ is sufficiently large parameter to be determined, basically the length of amplification. $\mathcal{P}$ is the set of primes in $[L_0,2L_0]$ which are coprime to $N$ and to be used for amplifier. $g\in \PGL_n(\A)$ is in a fixed compact domain, in particular we can assume $g=\prod_{v} g_v$ where $g_v$ is in a fixed domain of $\PGL_n(\F_v)$ when $v$ is archimedean, and $g_v\in K_v$ when $v$ is finite. Let $\omega=\prod \omega_v \in C_c(\GL_n(\A))$ be the test function for the pre-trace formula, where 
$\omega_v$ is
\begin{enumerate}
\item fixed test function which is rapidly decaying outside a compact domain when $v=\infty$;
\item Hecke operator as given in \cite{BlomerMaga:a} when $v\in \mathcal{P}$;
\item At $p$ as given in \eqref{Eq:DefnofTestfun};
\item $\Char_{K_v}$ otherwise.
\end{enumerate}

Let $F$ be an automorphic form on $\PGL_n$ whose local component at $p$ is a minimal vector and spherical at all other places. 

Then as in \cite{Marshall:} we have the following amplified pre-trace inequality using Lemma \ref{Cor:formaldegMC}.
\begin{align}\label{Eq:pretraceIneq}
|\pi(\omega)F(g)|^2 \asymp d_\pi^2 |\mathcal{P}|^2|F(g)|^2&\leq \sum\limits_{\gamma\in M_n(\F)}(\omega*\omega^*)(g^{-1}\gamma g)\\
&\ll d_\pi(|\mathcal{P}|+\sum_{i=1}^{n}\sum_{q_1,q_2\in \mathcal{P}} \frac{1}{L_0^{(n-1)i}}\sum\limits_{\gamma\in S(q_1^i q_2^{(n-1)i})}|\omega_\infty\omega_p(g^{-1}\gamma g)|).
\end{align}
In the last inequality we have used  \cite{BlomerMaga:}, and
\begin{equation}
S(m)=\{\gamma\in \GL_n(\Z)| \det \gamma=m\}.
\end{equation}
Note that $\omega_\infty$ is rapidly decaying outside a compact domain and $g_\infty$ is also in a fixed compact domain. Using the Cartan decomposition for $\GL_n(\R)=Z KAK$, where $K=SO(n)$ and $A$ is the diagonal torus with $\det=1$, 
we can assume that $\gamma$ contributing to \eqref{Eq:pretraceIneq} has the property that when we write $\gamma=zk_1 a k_2$, $a\in A_0$ for some fixed compact domain $A_0$ of $A$.  Define for $a=\diag(a_1, a_2, \cdots, a_n)\in A$, $|a|=\max\{|a_i|\}$. Then $|A_0|:=\max\{|a|, a\in A_0\}$ is finite and fixed.
Also by Lemma \ref{Cor:torusMC}, $\omega_p$ is $p-$adically concentrated on some torus $U_\L(1)\subset K_p$, and after conjugation by $g_p\in K_p$, $\gamma_p$ should still be $p-$adically close to some general torus  $T\subset K_p$. Note that we shall not use any information on the exact embedding of $T$, but purely the fact $T$ is a torus. Then those $\gamma$'s contributing to \eqref{Eq:pretraceIneq} should lie in the following set.
\begin{equation}
S(m, T, \cc)=\{\gamma\in \GL_n(\Z)| \det \gamma=m, \gamma_\infty \in Z KA_0K, \gamma_p \equiv t\in T\mod{p^{\cc}}\}.
\end{equation}
Then
\begin{equation}\label{Eq:pretraceIneq2}
d_\pi |\mathcal{P}|^2|F(g)|^2\ll |\mathcal{P}|+\sum_{i=1}^{n}\sum_{q_1,q_2\in \mathcal{P}} \frac{1}{L_0^{(n-1)i}}|S(q_1^i q_2^{(n-1)i},T,\cc)|
\end{equation}
\begin{rem}
In \cite{BlomerMaga:} there are finer conditions for $\gamma \in S(m)$ concerning its 1-st and 2-nd determinantal divisor. But we don't need these additional conditions for our purpose.
\end{rem}

\subsection{Estimate for Hecke return}
We shall give a bound for $|S(m,T,\cc)|$ in this subsection by utilising the fact $\gamma\in S(m,T,\cc)$ is $p-$adically close to some torus $T$.
\begin{lem}\label{Lem:abelianstructure}
When $p^{\cc}\gg n m^2 |A_0|^4 $,  $S(m, T, \cc)$ is abelian.
\end{lem}
\begin{proof}
For any $\gamma_1,\gamma_2\in S(m, T, \cc)$, let $\beta=[\gamma_1,\gamma_2]=\gamma_1^{-1}\gamma_2^{-1}\gamma_1\gamma_2$. Then 
\begin{equation}\label{eq1:det}
\det(\beta)=1
\end{equation}
\begin{equation}\label{eq2:integralwithdenom}
 \beta\in \GL_n(\Z[\frac{1}{m^2}])
\end{equation}
\begin{equation}\label{eq3:infty}
 \beta_\infty\in ZKA_0^4K
\end{equation}
\begin{equation}\label{eq4:atp}
 \beta_p\equiv 1\mod{p^{\cc}}.
\end{equation}
Here $A_0^4$ is the smallest compact region containing $a^4$ for all $a\in A_0$ (can probably be made more precise). So by definition $|A_0^4|=|A_0|^4$.

We shall show that $\beta=1$ for $p^{\cc}$ large enough. From \eqref{eq2:integralwithdenom} and \eqref{eq4:atp}, we can write
\begin{equation}
\beta=\left(\begin{array}{cccc}
1+\frac{p^{\cc}}{m^2}u_{1,1}&\frac{p^{\cc}}{m^2}u_{1,2} &\cdots & \frac{p^{\cc}}{m^2}u_{1,n}\\
\frac{p^{\cc}}{m^2}u_{2,1} &1+\frac{p^{\cc}}{m^2}u_{2,2} &\cdots &\frac{p^{\cc}}{m^2}u_{2,n}\\
\vdots &\vdots &\ddots &\vdots \\
\frac{p^{\cc}}{m^2}u_{n,1} &\frac{p^{\cc}}{m^2}u_{n,2}&\cdots &1+\frac{p^{\cc}}{m^2}u_{n,n}
\end{array}\right)
\end{equation}
where $u_{i,j}\in \Z$. Note that $\det \beta=1$. So when we do Cartan decomposition for $\beta$, we can take $z=1$. On the other hand, every matrix in $KA_0^4K$ has  entries bounded by $n|A_0|^4$. 
Thus when $\frac{p^{\cc}}{m^2}\gg n|A_0|^4$, we must have $u_{i,j}=0$ for $\beta$ and $\beta=1$.
\end{proof}


By the previous lemma we get that $S(m, T, \cc)=S(m, T, \cc)\cap \T$ for some global etale algebra $\T$ (possibly different from $T$).  $\T$ is of degree at most $n$, as it's abelian and thus diagonalisable over algebraically closed field, while the diagonal torus of $\GL_n$ is at most $n$ dimensional. 

When localised at $l$,  $\T_l=\bigotimes\limits_{i=1}^{k} T_l^i$ where $T_l^i$ is a field extension over $\F_l$ of degree $n_i$ and $\sum\limits_{i=1}^{k} n_i\leq n$. In particular $k\leq n$. Then for $e\in \T_l^*, e=\otimes e_{i}$ for $e_i\in T_l^i$,
\begin{equation}
 \det e=\prod \Nm_{T_l^i/\F_l}(e_i).
\end{equation} 
To see this one can assume that after proper conjugation, $\T_l$ can be embedded into $\GL_n$ in such a way that $T_l^i\hookrightarrow \GL_{n_i}$ and $\otimes\GL_{n_i}$ is embedded into $\GL_n$ blockwise diagonally.
From this, one can see that  
\begin{equation}
 \det e=\prod \det e_i=\prod \Nm_{T_l^i/\F_l}(e_i).
\end{equation} 

Suppose now that $$m=\prod\limits_{j}l_j^{a_j}$$ over a finite subset of primes $\{l_j\}\subset\mathcal{P}$, $\T_{l_j}=\bigotimes\limits_{i=1}^{k_j}T_{l_j}^i$.
For each field $T_{l_j}^i$, let $\mathcal{Q}_{l_j}^i$ be the set of ideals in $T_{l_j}^i$.
\begin{defn}
Let $\tau$ be the following map. 
\begin{align}
\tau: S(m, T, \cc)\cap \T &\rightarrow \otimes_{j} \otimes_{i} \mathcal{Q}_{l_j}^i\\
\gamma &\mapsto \otimes_i (\gamma_{l_j}^i),\notag
\end{align}
where $\gamma_{l_j}=\otimes \gamma_{l_j}^i$ for $\gamma_{l_j}^i\in T_{l_j}^i$ and $(\gamma_{l_j}^i)$ is the principal ideal in $T_{l_j}^i$ generated by $\gamma_{l_j}^i$.

\end{defn}

\begin{lem}\label{Lem:fiberoflocalisation}
The fiber of $\tau$ is finite and absolutely bounded in terms of $n$ and $|A_0|$.
\end{lem}
\begin{proof}
Let $\gamma_1,\gamma_2\in S(m, T, \cc)\cap \T$  with $\tau(\gamma_1)=\tau(\gamma_2)$. Then for $\beta'=\gamma_1^{-1}\gamma_2$, we have 
\begin{equation}\label{eq:beta'1}
\det(\beta')=1,
\end{equation}
\begin{equation}\label{eq:beta'2}
 \beta'_\infty\in ZKA_0^2K,
\end{equation}
\begin{equation}\label{eq:beta'3}
 \beta'\in \GL_n(O_v)
\end{equation}

for any finite place $v\neq l_j$. At $l_j$, we have that 
\begin{equation}\label{eq:beta'4}
 \beta'_l=\otimes_i \beta_{l_j}^{'i}\text{\ for }\beta_{l_j}^{'i}\in O_{T_{l_j}^i}^*
\end{equation} 
as  $\tau(\gamma_1)=\tau(\gamma_2)$. Let $f(x)$ be the characteristic polynomial of $\beta'$. 
From \eqref{eq:beta'3} and \eqref{eq:beta'4}, all the coefficients of $f(x)$ has to be integral as the localisation of $\beta'$ is either integral in $\GL_n(\F_v)$ or in $\otimes_i T_{l_j}^i$. 
\eqref{eq:beta'1} and \eqref{eq:beta'2} implies that all the coefficients of $f(x)$ are bounded by a polynomial of $|A_0|$ in terms of $n$. The possible solutions inside an etale algebra to all such characteristic polynomials can be absolutely bounded in terms of $n$ and $|A_0|$. 
\end{proof}
\begin{rem}
One can try to bound the fiber more effectively, which might be useful for a hybrid bound.
\end{rem}

\begin{lem}\label{Lem:imageoflocalisation}
 $$\# \Image(\tau) \leq \prod_j P(a_j, n), \text{\ and accordingly } |S(m, T, \cc)|\ll_{n,|A_0|}\prod_j P(a_j, n).$$
Here $P(a,n)=C^{n-1}_{n+a-1}$ is the number of partitions of $a_j$ into $n$ non-negative ordered integers. 

In particular if $\sharp j\ll_n 1$, $a_j\ll_n 1$, then $|S(m, T, \cc)|\ll_{n,|A_0|} 1$.
\end{lem}
\begin{proof}
First of all by $\det(\gamma)=m=\prod_j l_j^{a_j}$ and $\det(\gamma_{l_j})=\prod \Nm(\gamma_{l_j}^i)$, we have that $\prod \Nm ((\gamma_{l,i}))=l_j^{a_j}$.
As the ring of integers $O_{T_{l_j}^i}$ is a discrete valuation ring, it's automatically a PID, so any ideal in $\mathcal{Q}_{l_j}^i$ is of form $(\varpi_{T_{l_j}^i}^n u)=(\varpi_{T_{l_j}^i}^n)$ where $\varpi_{T_{l_j}^i}$ is a local uniformizer and $u$ is any unit in the local field. Its norm only depends on $n$, so there is at most one ideal in $\mathcal{Q}_{l_j}^i$ with given norm. Each such norm is a power of $l_j$.

Thus the local components of elements in $\Image(\tau)$ at $l_j$ are in injection to ways of writing $l_j^{a_j}$ as product of norms of ideals from $T_{l_j}^i$. Thus
\begin{equation}
 \# \Image(\tau) \leq \prod_j P(a_j, k_j)\leq \prod_j P(a_j,n).
\end{equation}
When $\sharp j\ll_n 1$, $a_j\ll_n 1$, $P(a_j,n)=C^{n-1}_{n+a_j-1}\ll_n 1$ and $$ |S(m, T, \cc)|\ll_{n,|A_0|}\prod_j P(a_j, n)\ll_n 1.$$
\end{proof}

\subsection{proof of the main result}
To apply Lemma \ref{Lem:abelianstructure}, we pick $L_0$ such that
\begin{equation}
 L_0^{2n^2}\asymp_{A_0,n} p^{\cc}, \text{\ i.e., } L_0\asymp_{A_0,n} p^{\frac{\cc}{2n^2}}
\end{equation}
Then by Lemma \ref{Lem:fiberoflocalisation}, \ref{Lem:imageoflocalisation}, \eqref{Eq:pretraceIneq2} now reads
\begin{align}
 d_\pi |\mathcal{P}|^2|F(g)|^2&\ll |\mathcal{P}|+\sum_{i=1}^{n}\sum_{q_1,q_2\in \mathcal{P}} \frac{1}{L_0^{(n-1)i}}|S(q_1^i q_2^{(n-1)i},T,\cc)|\\
 &\ll_{|A_0|,n} |\mathcal{P}|+\sum_{i=1}^{n}\sum_{q_1,q_2\in \mathcal{P}}  \frac{1}{L_0^{(n-1)i}}\notag\\
 &\ll_{|A_0|,n}  L_0.  \notag
\end{align}
Then by Lemma \ref{Lem:Conductorofpi}, \ref{Cor:formaldegMC} and \eqref{Eq:asypmCC}
\begin{equation}
|F(g)|\ll_{|A_0|,n} \frac{1}{(d_\pi L_0)^{1/2}}\asymp_{|A_0|,n,p} p^{\frac{c(n^2-n)}{4}-\frac{\cc}{4n^2}}\asymp C(\pi)^{\frac{n-1}{4}-\frac{1}{8n^3}}.
\end{equation}

\end{document}